\documentclass[12pt, french]{amsart} 
\usepackage{pythonhighlight}
\usepackage{amsmath, amssymb}
\usepackage{amsthm} 
\makeatletter
\newcommand{\addresseshere}{%
  \enddoc@text\let\enddoc@text\relax
}
\makeatother

\theoremstyle{plain} 
\newtheorem{thm}{\indent\sc Theorem}[section]
\newtheorem{lemma}[thm]{\indent\sc Lemma}
\newtheorem{corollary}[thm]{\indent\sc Corollary}
\newtheorem{proposition}[thm]{\indent\sc Proposition}
\theoremstyle{definition} 

\newtheorem{rmk}[thm]{\indent\sc Remark}

\theoremstyle{remark}
\newtheorem*{question}{\indent\sc Question}

\usepackage{enumitem}
\usepackage[utf8]{inputenc}
\usepackage{amsmath}%
\usepackage{caption}
\usepackage{amsfonts}%
\usepackage{amssymb}%
\usepackage{graphicx}
\usepackage{csquotes}
\usepackage{xcolor}
\usepackage{amsthm}
\usepackage{lipsum}
\usepackage{tikz}
\usepackage{datetime}
\usetikzlibrary{matrix}
\usetikzlibrary{cd}
\usepackage{float}
\restylefloat{table}

\def {\'e}
\def {\`e}
\def {\`a}
\def {\`u}
\def {\^e}
\def {\^a}
\def {\^o}
\def {\^{\i}}
\def {\"{\i}}
\def {\^u}
\def {\c c}
\def {\"e}
\def Ï{\oe}
\def Ë{\`A}
\catcode`\:=13
\def :{~\string:}
\catcode`\;=13
\def ;{~\string;}
\catcode`\!=13
\def !{~\string!}

\DeclareMathOperator{\rk}{rank}

\newdimen\step   \step=13.5pt
\newdimen\auxstep  \newdimen\another  \newdimen\dimaux 

\def\monte#1{\raise\step\hbox{#1}}
\def\descend#1{\raise-\step\hbox{#1}}

\def\montepeuc#1{\raise 3.8pt\hbox{#1}}
\def\montepeuC#1{\raise 5.8pt\hbox{#1}}

\def\descendpeuc#1{\raise -3.8pt\hbox{#1}}
\def\descendpeuC#1{\raise -5.8pt\hbox{#1}}

\def\recule#1{\auxstep=\step \multiply\auxstep by #1 \kern-\auxstep}
\def\avance#1{\auxstep=\step \multiply\auxstep by #1 \kern\auxstep}

\def\Place#1#2#3#4%
{\ifnum#1=0 \auxstep=#3pt \else \auxstep=#4pt\fi%
\another=\auxstep \kern\auxstep \auxstep=\step \advance\auxstep by -\another%
\ifnum#2=0 \another=#3pt \else \another=#4pt\fi%
\advance\auxstep by -\another\relax}

\def\Replace{\kern\auxstep\kern\another\relax}

\def\hbx{\hbox to 0pt}

\def\trait#1#2#3{\hbx{\vrule height #1pt depth #2pt width #3\hss}}

\newbox\Boitecercle                      
\setbox\Boitecercle=%
\hbx{\kern-1.75pt%
\raise-0.1pt\hbox{\kern-0.75pt\hbx{$\circ$\hss}\kern4.25pt}\kern-1.75pt}

\def\cercle{\copy\Boitecercle}
\def\boule{\monte{\copy\Boitecercle}}    

\newbox\Boiterond                        
\setbox\Boiterond=%
\hbx{\kern-1.75pt%
\hbox{\kern-0.75pt\hbx{$\bullet$\hss}\kern4.25pt}\kern-1.75pt}

\newbox\BoiteCercle
\setbox\BoiteCercle=%
\hbx{\kern-3.1389pt%
\hbox{\kern-0.75pt\hbox{$\odot$\hss}\kern-0.75pt}\kern-3.1389pt}

\def\Cercle{\hbox{\copy\BoiteCercle\copy\Boiterond}}

\def\Dessinetv#1#2%
{\auxstep=2.2222pt \dimaux=\step \advance\dimaux by\auxstep%
\ifnum#1=0 \another=1.75pt \else\another=3.1389pt\fi%
\advance\auxstep by\another \advance\dimaux by -\auxstep%
\ifnum#2=0 \another=1.75pt \else \another=3.1389pt\fi%
\advance\dimaux by-\another \kern-0.125pt%
\raise\auxstep\hbox{\vrule height\dimaux depth 0pt width 0.25pt}%
\kern-0.125pt}

\def\tvcc{{\Dessinetv00}}

\def\Dessineth#1#2%
{\Place{#1}{#2}{1.75}{3.1389}%
\raise 2.2222pt\trait{0.125}{0.125}{\auxstep}\Replace}

\def\thcc{{\Dessineth00}}    \def\thCc{{\Dessineth10}}
\def\thcC{{\Dessineth01}}    \def\thCC{{\Dessineth11}}

\def\DessineTh#1#2%
{\Place{#1}{#2}{1.75}{3.1389}%
\raise 2.2222pt\trait{0.35}{0.35}{\auxstep}\Replace}

\def\Dessinethh#1#2%
{\Place{#1}{#2}{1.55}{3.00}%
\raise3.2222pt\trait{0.125}{0.125}{\auxstep}%
\raise1.2222pt\trait{0.125}{0.125}{\auxstep}\Replace}

\def\Dessinethhh#1#2%
{\Place{#1}{#2}{1.25}{2.8}%
\raise 3.4444pt\trait{0.125}{0.125}{\auxstep}%
\raise 1.0000pt\trait{0.125}{0.125}{\auxstep}\Replace}

\def\thhhcc{\hbox{\Dessinethhh00\kern-\step\thcc}}
\def\thhhcC{{\Dessinethhh01\kern-\step\thcC}}
\def\thhhCc{{\Dessinethhh10\kern-\step\thCc}}
\def\thhhCC{{\Dessinethhh11\kern-\step\thCC}}

\def\Dessinethhhh#1#2%
{\Place{#1}{#2}{1.25}{2.9}%
\raise 3.4222pt\trait{0.125}{0.125}{\auxstep}%
\raise 1.0222pt\trait{0.125}{0.125}{\auxstep}%
\Replace\kern-\step%
\Place{#1}{#2}{1.5}{3.0}%
\raise 2.6222pt\trait{0.125}{0.125}{\auxstep}%
\raise 1.8222pt\trait{0.125}{0.125}{\auxstep}%
\Replace}

\def\Dessinepoints#1#2%
{\Place{#1}{#2}{1.75}{3.1389}%
\divide\auxstep by 4%
\kern\auxstep\kern-1.1pt\hbx{$\cdot$\hss}\kern 1.1pt%
\kern\auxstep\kern-1.1pt\hbx{$\cdot$\hss}\kern 1.1pt%
\kern\auxstep\kern-1.1pt\hbx{$\cdot$\hss}\kern 1.1pt%
\kern\auxstep\kern\another}

\def\Dessinef#1#2#3%
{\Place{#2}{#1}{1.75}{3.1389}\Replace\kern-\step%
\dimaux= 7pt \advance\auxstep by-\dimaux \divide\auxstep by 2%
\advance\auxstep by \another \advance\auxstep by \dimaux%
\raise 0.68pt\hbx{\kern-\auxstep$\scriptstyle #3$\hss}}

\def\Dessineff#1#2#3%
{\hbx{\kern-1pt\Dessinef{#1}{#2}{#3}\kern2pt\Dessinef{#1}{#2}{#3}\hss}}

\def\Dessinefff#1#2#3%
{\hbx{\kern-1.5pt\Dessinef{#1}{#2}{#3}%
\kern1.5pt\Dessinef{#1}{#2}{#3}\kern1.5pt\Dessinef{#1}{#2}{#3}\hss}}

\def\Dessinefourche#1#2#3
{\hbox{\kern #1pt%
\hbx{$\displaystyle <$\hss}\kern8pt%
\ifnum#2=1{\descendpeuC\Cercle}\else{\descendpeuc\cercle}\fi%
\ifnum#3=1{\montepeuC\Cercle}\else{\montepeuc\cercle}\fi}}

\def\ctvb{\cercle\tvcc\boule}   
                   \let\th=\thcc
\let\thh=\thhcc                 
             
\let\points=\pointscc
\let\fourche=\cfourchecc

\let\descendpeu=\descendpeuc    \let\montepeu=\montepeuc
\let\fgauche=\fgauchecc         

\let\fdroite=\fdroitecc

\newbox\aux

\def\gnoteN#1#2%
{\setbox\aux=\hbox{$\scriptscriptstyle #1$}%
\dimaux=\wd\aux \divide\dimaux by 2%
\kern-\dimaux\raise #2pt\hbx{$\scriptscriptstyle #1$\hss}\kern\dimaux}

\def\noteNC#1{\gnoteN{#1}{6.5}}

\def\noteSC#1{\gnoteN{#1}{-4.5}}

\def\gnoteE#1#2#3%
{\kern #2pt\raise #3pt\hbx{$\scriptscriptstyle #1$\hss}\kern-#2pt}

\def\gnoteO#1#2#3%
{\setbox\aux=\hbox{$\scriptscriptstyle #1$}%
\kern-#2pt\kern-\wd\aux\raise #3pt\hbx{$\scriptscriptstyle #1$\hss}%
\kern #2pt\kern\wd\aux}

\newbox\Boitedisque                      
\setbox\Boitedisque=%
\hbx{\kern-1.75pt%
\raise-0.1pt\hbox{\kern-0.75pt\hbx{$\bullet$\hss}\kern4.25pt}\kern-1.75pt}
\def\disque{\copy\Boitedisque}

\def\DessinefourcheD#1#2#3
{\hbox{\kern #1pt%
\hbx{$\displaystyle <$\hss}\kern8pt%
\ifnum#2=1{\descendpeuC\disque}\else{\descendpeuc\disque}\fi%
\ifnum#3=1{\montepeuC\disque}\else{\montepeuc\disque}\fi}}

\let\fourcheD=\cfourcheccD      

\def\Dessinefourched#1#2#3
{\hbox{\kern #1pt%
\hbx{$\displaystyle <$\hss}\kern8pt%
\ifnum#2=1{\descendpeuC\disque}\else{\descendpeuc\disque}\fi%
\ifnum#3=1{\montepeuC\cercle}\else{\montepeuc\cercle}\fi}}

\let\fourched=\cfourcheccd

\title[Projection of root systems]{%
Projection of root systems and the generalized injectivity conjecture for exceptional groups}
\author{Sarah Dijols}
\address{Sarah Dijols, PIMS Earth Sciences Building, 2207 Main Mall, Vancouver, BC V6T 1Z4, email: sarah.dijols@math.ubc.ca}

\begin{document}

\date{} 

\newpage
\maketitle


\begin{abstract}
Let $a$ be a real euclidean vector space of finite dimension and $\Sigma$ a root system in $a$ with a basis $\Delta$. Let $\Theta \subset \Delta$ and $M = M_{\Theta}$ be a standard Levi of a reductive group $G$ such that $a_{\Theta} = a_M\slash a_G$. Let us denote $d$ the dimension of $a_{\Theta}$, i.e the cardinal of $\Delta - \Theta$ and $\Sigma_{\Theta}$ the set of all non-trivial projections of roots in $\Sigma$. We obtain conditions on $\Theta$ such that $\Sigma_{\Theta}$ contains a root system of rank $d$. When considering the case of $\Sigma$ of type exceptional, we give a list of all exceptional root systems that can occur in $\Sigma_\Theta$ and use it to prove the generalized injectivity conjecture in most exceptional groups cases. 
\end{abstract}


\section{Introduction}

Let $a$ be a real euclidean vector space of finite dimension and $\Sigma$ a root system in $a$ with a basis $\Delta$. Let $\Theta \subset \Delta$, to avoid trivial cases we assume $\Theta$ is a proper subset of $\Delta$, i.e. that $\Theta$ is neither empty nor equal to $\Sigma$. \\
Let us consider the projection of $\Sigma$ on $a_{\Theta}$ (projection orthogonal to $\Theta$) and we denote $\Sigma_{\Theta}$ the set of all non-trivial projections of roots in $\Sigma$.
Let us denote $d$ the dimension of $a_{\Theta}$, i.e the cardinal of $\Delta - \Theta$. 

Let us also denote $\Delta_{\Theta}$ the set of projections of the simple roots in  $\Delta - \Theta$ on $a_{\Theta}$. 

In this article, we determine the conditions under which $\Sigma_{\Theta}$ contains a root system (for a subspace of $a_{\Theta}$) and what are the types of root system appearing. Since $a_{\Theta}$ is $d$-dimensional, such root system has a rank lower of equal to $d$. We will classify the systems of rank $d$ appearing when they exist. We then say they are of \emph{maximal rank}. 

\subsection{Motivations}

We consider $\textbf{G}$ a quasi-split reductive group over a local field $F$, and $\textbf{T}$ a maximal torus of $\textbf{G}$. As usual, the not-bold notation $G$ denotes the $F$-points of $\textbf{G}$. The context motivating the question in this article is the following: We fix $a = a_0^G := a_0\slash a_G$ quotient of the Lie algebra of the maximal $F$-split torus in a maximal torus $\textbf{T}$ by the Lie algebra of the maximal split torus $A_G$ in the center of $\textbf{G}$. We denote $\Sigma$ the root system of $G$ and $\Delta$ a basis of $\Sigma$. Let $M$ be a standard Levi subgroup of $G$ such that the set of simple roots in Lie($M$) is $\Delta^M = \Theta$. Then $a_{\Theta} = a_M\slash a_G$. \\
We do not consider the trivial case where $M= T$ and $M=G$. \\

This question emerges in an attempt to better understand the result of Silberger in \cite{silbergerSD}.
In Section 3.5 of his work, he claims that $$\Sigma_{\sigma}=\{\alpha\in \Sigma_{red}(A_{M})\vert \mu^{(M)_{\alpha}}(\sigma )=0\}$$ is a subset of $a_{M}^*$ which is a root system in a subspace of  $a_{M}^*$. We also refer to \cite{ope}[1.3] for a more modern treatment.  Here, $\sigma$ is a discrete series representation of a semi-standard Levi subgroup $M$ of a reductive group $G$ and $\mu^{(M)_{\alpha}}$ is one factor in the product formula of the $\mu$ function (see also \cite{wald}, Lemma V.2.1). Let us recall that the $\mu$ function is the main ingredient in the Plancherel measure, the unique Borel measure on the set of irreducible tempered representations of $G$ which was defined by Harish-Chandra to formulate the Plancherel formula for p-adic groups. Furthermore, Silberger expects that, for any semi-standard parabolic $P$ with Levi $M$ there is a unique subset of restriction of the simple roots of $\Sigma$ to $A_{M}$ which constitutes a basis of $\Sigma_\sigma$. Equivalently, we can talk about projections of roots in $\Sigma$ orthogonally to $\Theta$ if $M= M_{\Theta}$, i.e $\Delta^M = \Theta$. 
\newline

\begin{question}[Question 1]
Considering the restrictive definition of a root system, it is not immediately clear that the restrictions (resp. projections) of roots to $A_{M}$ constitute a root system of rank $|\Delta -\Theta|$; and in general it is not the case. The goal of this paper is to make precise the conditions on the semi-standard Levi $M$ (i.e on $\Theta$) such that $\Sigma(A_{M}):= \Sigma_{\Theta}$ contains a root system of maximal rank. This is the first question we will answer. 
\end{question}

\begin{question}[Question 2]
The second question we are addressing in this paper is the existence of a root system in $\Sigma_{\Theta}$ of maximal rank whose basis is constituted of projections of simple roots, i.e whose basis is $\Delta_\Theta$. This helps us understand better what root systems $\Sigma_\sigma$ occur for a given $\Sigma$.
\end{question}

In answering those questions, the cases of $\Sigma$ of classical type or exceptional type require a separate treatment. In the classical type case, our procedure is uniform and provide answers to both questions simultaneously. In this setting, our main result is the following:

\begin{thm} \label{mainlab}
Let $\Sigma$ be an irreducible root system of classical type (i.e of type $A,B,C$ or $D$) with basis $\Delta$. Let $\Theta$ be a subset of $\Delta$, and $d=|\Delta -\Theta|$. The systems contained in $\Sigma_{\Theta}$ are necessarily of classical type. In addition, if the irreducible (connected) components of $\Theta$ (seeing $\Theta$ in the Dynkin diagram) of type $A$ are all of the same length, and the interval between each of them of length one, then $\Sigma_{\Theta}$ contains an irreducible root system of rank $d$ (non necessarily reduced).
\end{thm}

Further, we deal with irreducible root systems of exceptional type. Then the main question to address is to identify if some exceptional type root systems of maximal rank appear in the projection. We answer this question using SageMath \cite{sagemath}. 
We do not attempt to list all the root systems of classical type and maximal rank 
appearing in the projection, as they are too many; further, as opposed to the case of $\Sigma$ of classical type, there is no uniform way to obtain such a list. However, our code can easily be adapted to check whether such and such classical root systems of maximal dimension occur in those projections. 

\begin{thm}\label{exc}
Let $\Sigma$ be an irreducible root system of exceptional type, with basis $\Delta$. We take the convention of numbering simple roots as in Bourbaki (and SageMath). Let $\Theta$ be a subset of $\Delta$ and $d=|\Delta -\Theta|$. In the following cases, and only for those cases, $\Sigma_{\Theta}$ contains an irreducible root system of rank $d$ of exceptional type.
\begin{itemize}
\item Let $\Sigma$ be of type $E_8$. If $\Theta$ contains one (any) root, there is one root system of type $E_7$ in $\Sigma_{\Theta}$. If $\Theta$ is either of those sets of roots $\{\alpha_2,\alpha_4\}, \{\alpha_1,\alpha_3\},$ or $\{\alpha_7,\alpha_8\}$ then $\Sigma_{\Theta}$ contains a system of type $E_6$. 
\item Let $\Sigma$ be of type $E_8$. If $\Theta=\{\alpha_2,\alpha_3,\alpha_4,\alpha_5\}$, then $\Sigma_{\Theta}$ contains a system of type $F_4$. If $\Theta=\{\alpha_1,\alpha_2,\alpha_3,\alpha_4,\alpha_5,\alpha_6\}$, $\Sigma_{\Theta}$ contains a system of type $G_2$.
\item Let $\Sigma$ be of type $E_7$. If $\Theta=\{\alpha_2,\alpha_5,\alpha_7\}$, then $\Sigma_{\Theta}$ contains a system of type $F_4$. If $\Theta$ is constituted of the simple roots $\{\alpha_2,\alpha_4,\alpha_5,\alpha_6,\alpha_7\}$ or $\{\alpha_1,\alpha_2,\alpha_3,\alpha_5,\alpha_7\}$ then $\Sigma_{\Theta}$ contains the root of system of type $G_2$. 
\item If $\Sigma$ is of type $E_6$ and $\Theta$ is constituted of the simple roots $\{\alpha_1,\alpha_3,\alpha_5,\alpha_6\}$ then $\Sigma_{\Theta}$ contains the root system of type $G_2$. 
\item If $\Sigma$ is of type $F_4$ and $\Theta$ contains either of those sets of roots $\{\alpha_1,\alpha_2\}$ and $\{\alpha_3,\alpha_4\}$ or the root of system of type $G_2$.
\end{itemize}
\end{thm}

To answer our second question, we also check which among the list above (Theorem \ref{exc}) have a basis constituted of projections of simple roots, i.e whose basis is $\Delta_\Theta$. The answer is as follows:

\begin{thm}\label{exc2}
Let $\Sigma$ be an irreducible root system of exceptional type with basis $\Delta$. Let $\Theta$ be a subset of $\Delta$ and $d=|\Delta -\Theta|$. In the following cases, and only for those cases, $\Sigma_{\Theta}$ contains an irreducible root system of rank $d$ and of exceptional type whose basis is constituted of projections of simple roots in $\Delta$. 
\begin{itemize}
\item Let $\Sigma$ be of type $E_8$. If $\Theta=\{\alpha_2,\alpha_3,\alpha_4,\alpha_5\}$, then $\Sigma_{\Theta}$ contains a system of type $F_4$. If $\Theta=\{\alpha_1,\alpha_2,\alpha_3,\alpha_4,\alpha_5,\alpha_6\}$, $\Sigma_{\Theta}$ contains a root system of type $G_2$.
\item Let $\Sigma$ be of type $E_7$. If $\Theta=\{\alpha_2,\alpha_5,\alpha_7\}$, then $\Sigma_{\Theta}$ contains a system of type $F_4$. If $\Theta$ is constituted of the simple roots $\{\alpha_2,\alpha_4,\alpha_5,\alpha_6,\alpha_7\}$ or $\{\alpha_1,\alpha_2,\alpha_3,\alpha_5,\alpha_7\}$ then $\Sigma_{\Theta}$ contains a root of system of type $G_2$. 
\item If $\Sigma$ is of type $E_6$ and $\Theta$ is constituted of the simple roots $\{\alpha_1,\alpha_3,\alpha_5,\alpha_6\}$ then $\Sigma_{\Theta}$ contains the root of system of type $G_2$. 
\item If $\Sigma$ is of type $F_4$ and $\Theta$ is constituted of either $\{\alpha_1,\alpha_2\}$ or $\{\alpha_3,\alpha_4\}$, then $\Sigma_{\Theta}$ contains a root of system of type $G_2$. 
\end{itemize}
\end{thm}

Finally, we also consider the possibility of an exceptional root system to occur as part of a product of root systems appearing in $\Sigma_\Theta$ (i.e the root system appearing is reducible) in Proposition \ref{exc3}. Once more, this is a case-by-case analysis with (fortunately) only a limited set of cases to consider. 
\newline

An important application of those results is the Generalized Injectivity Conjecture, a conjecture formulated by Casselman and Shahidi in \cite{sha}. 
In \cite{GICD}, we have proved:
\begin{thm}[Generalized Injectivity for quasi-split group, \cite{GICD}] \label{mainresult2}
Let $G$ be a quasi-split, connected group defined over a $p$-adic field $F$ (of characteristic zero). 
Let $\pi_0$ be the unique irreducible generic subquotient of the standard module $I_P^G(\tau_{\nu})$, let $\sigma$ be an irreducible, generic, cuspidal representation of $M_1$ such that a twist by an unramified real character of $\sigma$ is in the cuspidal support of $\pi_0$.
Suppose that all the irreducible components of $\Sigma_{\sigma}$ are of type $A,B,C$ or $D$, then, under certain conditions on the Weyl group of $\Sigma_{\sigma}$, $\pi_0$ embeds as \emph{a 
subrepresentation} in the standard module $I_P^G(\tau_{\nu})$.
\end{thm}

As a result of our exceptional root systems calculations, we can observe that the Generalized Injectivity conjecture holds for exceptional groups in all cases where the inducing standard parabolic subgroup $P_\Theta$ (with Levi $M_\Theta$) is indexed by $\Theta$-subsets different from the list given in Theorem \ref{exc2} and Proposition \ref{exc3}. 

\begin{corollary}[Corollary of Theorem \ref{exc2} and Proposition \ref{exc3}]
Let $\Sigma$ be the root system of an exceptional group $G$ over a $p$-adic field, and $\Theta$ a subset of its basis $\Delta$. Define $P_\Theta$ the standard parabolic subgroup of $G$ with standard Levi $M_\Theta$. The generalized injectivity conjecture holds for all standard modules $I_{P_{\Theta}}^G(\tau_{\nu})$ outside of the following set of pairs $(\Sigma, \Theta)$.
\begin{itemize}
\item If $\Sigma$ is of type $E_8$ and $\Theta=\{\alpha_2,\alpha_3,\alpha_4,\alpha_5\}$, or $\Theta=\{\alpha_1,\alpha_2,\alpha_3,\alpha_4,\alpha_5,\alpha_6\}$. Also if $\Theta$ is either of $\{\alpha_1,\alpha_3,\alpha_5, \alpha_6, \alpha_8\}$, $\{\alpha_2,\alpha_4,\alpha_5, \alpha_6, \alpha_7\}$, $\{\alpha_2, \alpha_5,\alpha_7\}$ or $\{\alpha_1,\alpha_3,\alpha_5, \alpha_6\}$.
\item If $\Sigma$ is of type $E_7$ and $\Theta$ is either of $\{\alpha_2,\alpha_4,\alpha_5,\alpha_6,\alpha_7\}$, $\{\alpha_1,\alpha_2,\alpha_3,\alpha_5,\alpha_7\}$ or $\{\alpha_1,\alpha_3,\alpha_5, \alpha_6\}$.
\item If $\Sigma$ is of type $E_6$ and $\Theta=\{\alpha_1,\alpha_3,\alpha_5,\alpha_6\}$.
\item If $\Sigma$ is of type $F_4$ and $\Theta$ is constituted of either $\{\alpha_1,\alpha_2\}$ or $\{\alpha_3,\alpha_4\}$. 
\end{itemize}
\end{corollary}

The condition of \emph{maximal rank} of $\Sigma_{\sigma}$ is also crucial to the existence of a discrete series subquotient in the induced module $I_{P_1}^G(\sigma_{\lambda})$ whenever $\lambda \in a_{M}^*$ is known as a residual point, as studied in \cite{heiermann}. Heiermann's approach to the infinitesimal character of an irreducible discrete series requires the notion of residual point which itself requires the rank of $\Sigma_{\sigma}$ to be maximal (see Definition 2.1 in \cite{GICD}), see also Section 3.8 in \cite{silbergerSD}.
The conditions we have obtained, in the classical root system case, on the form of the Levi $M$, in order to obtain a maximal rank root system, have already been implicitly used in the literature, see for instance Proposition 1.13 in \cite{ope}.

\subsection{Method} \label{method}
 Of course, there are always subsystems of rank 1 and as $\Theta$ is assumed to be non-empty there is no need to discuss the case where $\Sigma$ is of rank 2 (in particular $G_2$).
We start with the description of our method in the case of classical root systems. We will therefore consider the root systems $\Sigma$ of rank $ n\geq 3$ and $d \leq n-2$. Let us remark that we will find irreducible non reduced root systems: they are the $BC_d$ which contain three subsystems of rank $d$: $B_d$, $C_d$ and $D_d$.

We will use the following remark (see \cite[Equation (10) in VI.3, Proposition 12 in VI.4, Chapter VI]{bourbaki}). Let $\alpha$ and $\beta$ be two non-orthogonal distinct elements of a root system.
Set $$C=\left(\frac{1}{\cos({\alpha},{\beta})}\right)^2\qquad\hbox{and}\qquad R=\frac{||{\alpha}||^2}{||{\beta}||^2}\,\,.$$
The only possible values for $C$ (the inverse of the square of the cosinus of the angle between two roots) are 4, 2 and $4/3$ whereas assuming the length of $\alpha$ larger or equal to the one of $\beta$, the quotient of the length is respectively 1, 2 or 3. Thus, if $||\alpha||\ge||\beta||$
$$ \frac{C}{R}\in\{2^2,1,(2/3)^2\}\qquad\hbox{and}\qquad CR=4\,\,.$$

We will therefore compute the quotient of lengths and the angles of the non-trivial projections of roots in $\Sigma$, in particular those of elements in $\Delta - \Theta$.

%

In general $\Sigma_{\Theta}$ is not a root system, however let us observe:
\begin{lemma} \label{uniqueway}
The elements in $\Sigma_{\Theta}$  are, in a unique way, linear combination with entire coefficients all with the same sign of the elements in $\Delta_{\Theta}$.
\end{lemma}
From Theorem 3 (page 156) or Corollary 3 (page 162) in Chapter 6, \S1, Sections 6 and 7 in \cite{bourbaki}; we know any root in $\Sigma$ can be written in a unique way as linear combination with entire coefficients all with the same sign of the elements in the basis $\Delta$. Then the statement in Lemma \ref{uniqueway} follows since the projection orthogonal to any subset $\Theta \subset \Delta$ (i.e projection onto $W^{\perp}$, if $W$ is the vector space generated by $\Theta$) is a linear map.

\section*{Acknowledgements}
The method used in the context of classical groups is due to Jean-Pierre Labesse. The author warmly thanks him for communicating it to us and for many very helpful discussions on this work. He also shared with us his note on a basis of root systems of type $E$ also used in \cite{JIMJ}. This work was started at the end of the author's PhD thesis at Aix-Marseille University and paused during a few years due to technical hurdles with the SageMath code.
We thank Olivier Ramar\'e for useful advices which got us started on SageMath, and Bill Casselman who pointed out an error in the code and suggested how to correct it. We are also grateful to him for a very enthusiastic discussion which led to some simplifications of the script. We also thank Samuel Leli\`evre, and Danny Glin who respectively helped us estimate the complexity of some earlier versions of Question 2's code and get access to the University of Calgary' server to run it.  

\section{Classical root systems}

In this section, we prove Theorem \ref{mainlab} \textit{via} a case-by-case analysis.

\subsection{The case where $\Sigma$ is of type $A_n$}
Let us consider $a_0$ to be of dimension $n+1$ with an orthonormal basis $e_1, e_2, \ldots ,e_{n+1}$. Let us denote $\Xi$ this ordered basis, i.e the ordered set of the $e_i$. The elements of $\Sigma$ are the $e_i - e_j$ with $i \neq j$; they generate a subspace $a$ of dimension $n$; let $\Delta$ be the set of simple roots $\alpha_i = e_i - e_{i+1}$. Let us denote $\overline{e_i}$ the projection of $e_i$ on $a_{\Theta}$. The Dynkin diagram of $\Theta$ is a union of irreducible (or connected) components of type $A$. Therefore, the data of $\Theta$ corresponds to a partition of the ordered set $\Xi$ in a disjoint (ordered) union of ordered parts that we index by the smallest index appearing in the indices of the basis vectors associated: 
$$\Xi=\Xi_1\cup\cdots\cup\Xi_{l}\,\,.$$
The correspondence is defined as follows, the part:
$$\Xi_r=\{e_r,\cdots,e_{r+m}\}$$
is associated to the empty subset if $m=0$ and to the subset of simple roots 
$$\{\alpha_r,\cdots,\alpha_{r+m-1}\}\qquad\hbox{if $m\ge1$}\,\,.$$ 
Let us consider an element $e_i$ in the basis $\Xi$ of $a_0$. Let $r$ be the smallest integer $j$ such that $\overline{e_j} = \overline{e_r}$, and let $r+m$ be the largest. We will have $\overline{e_k} = \overline{e_i}$ for any $k$ such that $r\le k\le r+m$. If $m=0$, it is clear. Observe that if $m=0$, the two simple consecutive roots $\alpha_{i-1}$ and $\alpha_{i}$ where $e_i$ appears are outside $\Theta$. Now, let $m \geq 0$, the root $e_r - e_{r+m}$ has a trivial projection on $a_{\Theta}$ and therefore by Lemma \ref{uniqueway} all the simple roots that occur in the expression of this root shall be in $\Theta$. As a result, the roots $\alpha_k = e_k - e_{k+1}$ belong to $\Theta$ for any $k$ such that $r\le k\le r+m-1$ and we have:
$$\overline{e_k}=\frac{e_r+e_{r+1}+\cdots+e_{r+m}}{m+1}$$ for all $k$ such that $r\le k\le r+m$.
Indeed, this expression of $\overline{e_k}$ is then orthogonal to all the roots $\alpha_k = e_k - e_{k+1}$ for any $k$ such that $r\le k\le r+m-1$.

Such a chain of simple roots is a connected component of length $m$ of the Dynkin diagram associated to $\Theta$. We have observed that such a connected component is empty when $e_r$ is orthogonal to all the elements in $\Theta$ in which case $m=0$ i.e the two consecutive simple roots $\alpha_{r-1}$ and $\alpha_r$ are outside $\Theta$. If $e_r$ is associated to a length $m$ connected component of $\Theta$ and therefore belongs to an ordered part of cardinal $m+1$ of $\Xi$, the square of the length of $\overline{e_r}$ is:
$$||\overline{e_r}||^2=\frac{1}{m+1}\,\,.$$ 
Let us consider three vectors $e_r, e_s$ and $e_t$ whose projections $\overline{e_r},\overline{e_s}$ and $\overline{e_t}$ are distinct and are associated to three components of $\Theta$ of type $A_m, A_p$ and $A_q$. Let $\alpha =e_i - e_j$ a root whose projection 
$$\overline{\alpha}=\pm(\overline{e_r}-\overline{e_s})\,\,.$$
$$||\overline{\alpha}||^2=\frac{1}{m+1}+\frac{1}{p+1}\,\,.$$
Let us consider a root $\beta=e_k-e_l$ whose projection is $$\overline{\beta}=\pm(\overline{e_s}-\overline{e_t})$$
we obtain
$$||\overline{\beta}||^2=\frac{1}{p+1}+\frac{1}{q+1}$$ and the square of the scalar product of $\overline{\alpha}$ and  $\overline{\beta}$ is
$$\left(<\overline{\alpha},\overline{\beta}>\right)^2=\frac{1}{(p+1)^2}\,\,.$$
Thus we have:
$$C=\left(\frac{1}{\cos(\overline{\alpha},\overline{\beta})}\right)^2=\left(1+\frac{p+1}{m+1}\right)\left(1+\frac{p+1}{q+1}\right),$$
and if we assume $||\overline{\beta}||\ge||\overline{\alpha}||$ i.e  $q\ge m$, we have:
$$R=\frac{||\overline{\alpha}||^2}{||\overline{\beta}||^2}=\frac{\left(1+\frac{p+1}{m+1}\right)}{\left(1+\frac{p+1}{q+1}\right)}.$$

Then
$$ \frac{C}{R}=\left(1+\frac{p+1}{q+1}\right)^2\in\{2^2,1,(2/3)^2\}\qquad\hbox{and}\qquad
CR=\left(1+\frac{p+1}{m+1}\right)^2=4
\,\,.$$
The only possible case is $C/R=4$ and thus $R=1$ and $C=4$. This implies
$m=p=q$
and $\{\overline{\alpha},\overline{\beta}\}$ generate a root system of type $A_2$: 
$\pm(\overline{e_r}-\overline{e_s})$, $\pm(\overline{e_s}-\overline{e_t})$
and $\pm(\overline{e_r}-\overline{e_t})$.

\begin{lemma}  
If $\Sigma$ is of type $A_n$ the only irreducible subsystems appearing in $\Sigma_\Theta$
are of type $A$. 
To have a root system of rank the dimension $d$ of  
$a_{\Theta}$ it is necessary if $d>1$, that the Dynkin diagram of $\Theta$ be a disjoint union of $d+1$ connected components of type $A_m$ with $m\ge0$,
the intervals between each such component being of length one:   $$n+1=(m+1)(d+1)$$ 
$$\underbrace{\disque\noteNC{\alpha_1}\th\disque\th\points\th\disque  \th\disque\noteNC{\alpha_m}}_{A_m} 
\th\cercle\th
\underbrace{\disque\th\disque\th\points\th\disque  \th\disque}_{A_m} \th\cercle\th
\points\th\cercle\th \underbrace{\disque\th\disque\th\points\th\disque 
\th\disque\noteNC{\alpha_n}}_{A_m} $$
This corresponds to a partition of the ordered basis
$\Xi$ in an union of $d+1$ ordered parts
of cardinal $m+1$:
$$\Xi=\Xi_1\cup\cdots\cup\Xi_{d+1}$$
where
$$\Xi_r=\{e_{(r-1)(m+1)+1}\cdots e_{r(m+1)}\}\,\,.$$
In this case $\Sigma_\Theta$ is of type $A_d$.
\end{lemma}
\begin{proof}
An irreducible subsystem is necessarily generated by the projections of roots of the form 
$\overline{\alpha}=\overline{e_i}-\overline{e_j}$ 
where the vectors $\overline{e_*}$ are all of the same length; when we order these vectors following the $d+1$ indices, we obtain a basis of a subspace $b_0$ of $a_0$ containing a subspace $b$ of codimension one in which the $\overline{e_i}-\overline{e_j}$ generate a system of type $A$. The rest of the corollary follows easily.
\end{proof}

\subsection{The case where $\Sigma$ is of type $B_n$}

In this case, the basis of $a$ is constituted of the $e_i$ for $i\in\{1,\cdots,n\}$ and the elements in $\Sigma$
are the $\pm e_i$ and the $\pm e_i\pm e_j$ and $\Delta$ is formed of the
$\alpha_i=e_i-e_{i+1}$ for $i\le n-1$ and of $\alpha_n=e_n$. The set $\Theta$ is an union of irreducible components which are all of type $A$ except for at most one which is of type $B_k$.

We distinguish two cases according to whether $e_n$ belongs to $\Theta$ or not, i.e according to whether one of the components is of type $B$ or not (case $k=0$).

Case 1 ($k=0$): $e_n\not\in\Theta$. In this case $\Theta$ is an union of components of type $A$.
As in the previous case, let us consider three vectors $e_r$, $e_s$ and $e_t$ whose non-trivial projections
$\overline{e_r}$, $\overline{e_s}$ ans $\overline{e_t}$ are distinct and associated to three components $\Theta$ of type $A_m$, $A_p$ and $A_q$.
Let us consider the roots of the form $\alpha=\pm e_i\pm e_j$ and $\beta=\pm e_k\pm e_l$
and let us suppose their projections write
$$\overline{\alpha}=\pm(\overline{e_r}\pm\overline{e_s})\qquad\hbox{and}\qquad
\overline{\beta}=\pm(\overline{e_s}\pm\overline{e_t})\,\,.$$
The projections are non-trivial, non-collinear, and non-orthogonal. The computations done in the previous subsection show that this family of vectors form a root system if and only if $m=p=q$. We also have in the projection of $\Sigma$ the vectors of the form:
$$\overline{\gamma}=\pm\overline{e_v}\qquad
\hbox{for $v\in\{r,s,t\}$}$$
Thus a system of type $B_3$.
Furthermore, $m \geq 1$, we also have in the projection of $\Sigma$, vectors of the form:
$$\overline{\delta}=\pm2\overline{e_v}\qquad
\hbox{for $v\in\{r,s,t\}$}$$ and in the end we obtain a root system of type $BC_3$.  

Let us consider now two roots $\alpha=\pm e_i\pm e_j$ and $\delta=\pm e_k$ whose projections write
$\overline{\alpha}=\pm(\overline{e_r}\pm\overline{e_s})$ and $\overline{\delta}=\pm\overline{e_s}$.  
We observe that
$$||\overline{\alpha}||^2=\frac{1}{m+1}+\frac{1}{p+1}
\qquad\hbox{and}\qquad
||\overline{\delta}||^2=\frac{1}{p+1}\,\,.$$
Further $||\overline{\alpha}||>||\overline{\delta}||$ and we have: $$\left(<\overline{\alpha},\overline{\delta}>\right)^2=\frac{1}{(p+1)^2}\,\,.$$

Therefore
$$C=\left(\frac{1}{\cos^(\overline{\alpha},\overline{\delta})}\right)^2=1+\frac{p+1}{m+1}
\qquad\hbox{and}\qquad
R=\frac{||\overline{\alpha}||^2}{||\overline{\delta}||^2}=1+\frac{p+1}{m+1}$$

So we have $C= R$ which forces $C= R=2$ and we recover the condition $m=p$.

Let us also remark that two short roots (that is of type $\pm\overline{e_r}$) 
or long (that is of type $\pm2\overline{e_r}$) (the length being relative to the length of roots
$\pm(\overline{e_s}\pm\overline{e_t})$) are necessarily proportional or orthogonal. This observation excludes the occurrence of a root system of type $F_4$.
Combining these observations, we see that except if $m=0$ (trivial case where the projection is the identity), we obtain maximal subsystems of type $BC$ (in particular non reduced).

Case 2 ($k\ge1$): $e_n\in\Theta$.  The projection on the orthogonal complement of $e_n$ gives a system $B_{n-1}$ and reiterating this procedure when $\Theta$ contains $B_k$, we recover the case 1 previously treated for $B_{n-k}$.
In conclusion, we have proven:

\begin{lemma}
The maximal subsystems are of type $B$ or $BC$. 
These contain the subsystems of type $B$, $C$ or $D$ of the same rank.
Let us assume $e_n$ belongs to a connected component of length $k$ (then of type $B_k$), with $k\ge0$ (the case  $k=0$ is the case in which $e_n$ does not belong to $\Theta$).
Then, the set $\Sigma_\Theta$ contains a system of rank equal to the dimension $d$ of 
$a_\Theta$ if the other components are all of the same length $m$ (and type $A_m$), the intervals between any of these components being of length one with $n-k=(m+1)d$. 
The projected system is of type  $BC_d$ except if $m=0$
in which case we obtain $B_{n-k}$.

Case 1: $k=0$, $n=d(m+1)$: the projected system is of type $BC_d$ if $m\ge1$.
$$\underbrace{\disque\noteNC{\alpha_1}\th\disque\th\points\th\disque  \th\disque\noteNC{\alpha_m}}_{A_m} 
\th\cercle\th
\underbrace{\disque\th\disque\th\points\th\disque  \th\disque}_{A_m} \th\cercle\th
\points\th\cercle\th \underbrace{\disque\th\disque\th\points\th\disque 
\th\disque}_{A_m} \thh\fdroite\cercle\noteNC{\alpha_n} $$
\par
Case 2: $k\ge1$, $n-k=d(m+1)$: the projected system is of type $BC_d$.
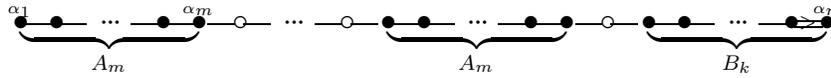
$$\underbrace{\disque\noteNC{\alpha_1}\th\disque\th\points\th\disque  \th\disque\noteNC{\alpha_m}}_{A_m} 
\th\cercle\th\points\th\cercle\th
\underbrace{\disque\th\disque\th\points\th\disque  \th\disque}_{A_m}
 \th\cercle\th
  \underbrace{\disque\th\disque\th\points\th
\disque \thh\fdroite\disque\noteNC{\alpha_n}}_{B_k} $$
This corresponds to a partition of the ordered basis $\Xi$ of cardinal $n$ in a union of $d+1$ ordered parts
 $$\Xi=\Xi_1\cup\cdots\cup\Xi_{d+1}$$
where
$$\Xi_r=\{e_{(r-1)(m+1)+1}\cdots e_{r(m+1)}\}
\qquad \hbox{for $1\le r\le d$ and}\qquad\Xi_{d+1}=\{e_{d(m+1)+1}\cdots e_{d(m+1)+r}\}.$$

\end{lemma}

\subsubsection{The case where $\Sigma$ is of type $C_n$}

In this case the basis of $a$ is formed with the $e_i$ for $i\in\{1,\cdots,n\}$ and the elements of $\Sigma$
are the $\pm 2e_i$ and the $\pm e_i\pm e_j$; moreover $\Delta$ is constituted of the
$\alpha_i=e_i-e_{i+1}$ for $i\le n-1$ and of $\alpha_n=2 e_n$. 
The set $\Theta$ is an union of irreducible components all of type $A$ except for at most one of type $C_k$.
We distinguish two cases whether $e_n$ belongs or not to $\Theta$.

Case 1 ($k=0$): $2e_n\not\in\Theta$. 
In this case $\Theta$ is an union of components of type $A$.
As in the case of $\Sigma$ of type $A_n$, let us consider three vectors $e_r$, $e_s$ and $e_t$ whose projections (which are non-zero) $\overline{e_r}$, $\overline{e_s}$ et $\overline{e_t}$ are distinct and associated to three components of $\Theta$
of type $A_m$, $A_p$ and $A_q$ and roots $\alpha=\pm e_i\pm e_j$ and $\beta=\pm e_k\pm e_l$
whose projections are 
$$\overline{\alpha}=\pm(\overline{e_r}\pm\overline{e_s})
\qquad\hbox{and}\qquad
\overline{\beta}=\pm(\overline{e_s}\pm\overline{e_t})$$
They will constitute a root system if and only if $m=p=q$.
Then we obtain a root system of type $C_3$ constituted of the  $\pm(\overline{e_r}\pm\overline{e_s})$, $\pm(\overline{e_s}\pm\overline{e_t})$,
$\pm(\overline{e_r}\pm\overline{e_t})$ and $\pm2\overline{e_v}$ for $v\in\{r,s,t\}$.

Let us now consider the two roots $\alpha=\pm e_i\pm e_j$ and $\beta=\pm 2e_k$ whose projections write
$$\overline{\alpha}=\pm\overline{e_r}\pm\overline{e_s}\qquad\hbox{and}\qquad
\overline{\beta}=\pm\overline{2e_s}\,\,.$$
$$||\overline{\alpha}||^2=\frac{1}{m+1}+\frac{1}{p+1}\qquad\hbox{and}\qquad
||\overline{\beta}||^2=\frac{4}{p+1}$$
and therefore $$\left(<\overline{\alpha},\overline{\beta}>\right)^2=\frac{4}{(p+1)^2}
\qquad\hbox{and}\qquad
C=\left(\frac{1}{\cos^(\overline{\alpha},\overline{\beta})}\right)^2=
\left(1+\frac{p+1}{m+1}\right)\,\,.$$
If we assume  $||\overline{\beta}||\ge||\overline{\alpha}||$ we have
$$R=\frac{||\overline{\beta}||^2}{||\overline{\alpha}||^2}=\frac{4}
{\left(1+\frac{p+1}{m+1}\right)}$$ and $CR=4$. All the cases are \textsl{a priori} possible.

If $C=2$ et $R=2$ then we necessarily have $p=m$.  
The vectors $\overline{\alpha}$ and $\overline{\beta}$ are the basis of a root system of a type $C_2$ where $\overline \beta$ is the long root.
The roots are
$$\pm\overline{\alpha}=\pm(\overline{e_r}-\overline{e_s})
\quad\hbox{,}\quad
\pm\overline{\beta}=\pm2\overline{e_s}
\quad\hbox{,}\quad
\pm(\overline{\alpha}+\overline{\beta})=\pm(\overline{e_r}+\overline{e_s})
\quad\hbox{and}\quad
\pm(2\overline{\alpha}+\overline{\beta})=\pm2\overline{e_r}\,\,.$$
The case $C=4$ and $R=1$ implies $$(p+1)=3(m+1)\qquad\hbox{and therefore}\qquad  p=3m+2.$$ 
Then $||\overline{\alpha}||$ and $||\overline{\beta}||$ constitute the basis of a 
a root system of type $A_2$ whose roots are $$\pm\overline{\alpha}=\pm(\overline{e_r}-\overline{e_s})\qquad\hbox{,}\qquad
\pm\overline{\beta}=\pm2\overline{e_s}\qquad\hbox{and}\qquad\pm(\overline{\alpha}+\overline{\beta})=\pm(\overline{e_r}+\overline{e_s})$$
but the vector $\pm2\overline{e_r}$ does not contribute to this system.

Finally if $C=4/3$ we have $$3(p+1)=(m+1)\qquad\hbox{and therefore} \qquad m=3p+2.$$
This forces $R=3$ which is a configuration of simple roots for a root system of type $G_2$ where $\overline \beta$ is the long root. However, $\Sigma_\Theta$ does not contain all the necessary roots for such a system; indeed the root
 $$\overline{\beta}+3\overline{\alpha}=3\overline{e_r}-\overline{e_s}$$ is not obtained.
 
Let us assume $||\overline\alpha||\ge||\overline\beta||$ we have $C/R=4$ and we recover the case $C=4$, $R=1$ and therefore
$(p+1)=3(m+1)$.

Case 2 ($k\ge1$): $e_n\in\Theta$.  The projection on the orthogonal complement of $e_n$ gives a system of type $BC_{n-1}$. And, reiterating this procedure, we recover the case of $BC_{n-k}$ which can be treated using our previous considerations on $B_{n-k}$ and $C_{n-k}$. 

To conclude, we have proved:

\begin{lemma}
The maximal subsystems are of type $A$,  $B$, $C$, $D$. 
Let us assume $2e_n$ belongs to a connected component of length $k$ (and type $C_k$), with $k\ge1$. 
The projection on the orthogonal of this component is a root system of type $BC_{n-k}$.
We recover the case where $k=0$, i.e where $e_n$ does not belong to $\Theta$ for a system of type $BC$.  

If $d\ge3$ the set $\Sigma_\Theta$ contains a system of rank equal to the dimension $d$
of $a_\Theta$ if the other components are all of the same length $m\ge0$ (and type $A_m$), the intervals between any of these components being of length one with $n-k=(m+1)d$, then the projected system is of type $BC_d$
(or $C_n$ if $k=0$ and $m=0$, trivial case excluded).

If $d=2$ we obtain either $BC_d$ when the two components of type $A$ are of length $m$
or $A_2$ when $(p+1)=3(m+1)$.

The case $k=0$, with $n=(m+1)d$ and projected system $C_d$
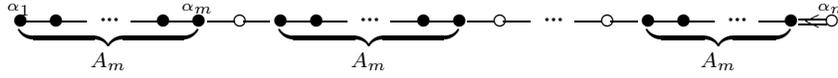
$$\underbrace{\disque\noteNC{\alpha_1}\th\disque\th\points\th\disque  \th\disque\noteNC{\alpha_m}}_{A_m} 
\th\cercle\th
\underbrace{\disque\th\disque\th\points\th\disque  \th\disque}_{A_m} \th\cercle\th
\points\th\cercle\th \underbrace{\disque\th\disque\th\points\th
\disque}_{A_m} \thh\fgauche\cercle \noteNC{\alpha_n}$$
The case $k=0$, with $p=3m+2$ and $n=4(m+1)$, and projected system containing $A_2$
$$ 
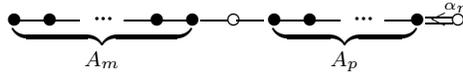
\underbrace{\disque\th\disque\th\points\th\disque  \th\disque}_{A_m} \th\cercle\th
\underbrace{\disque\th\disque\th\points\th
\disque}_{A_{p}} \thh\fgauche\cercle \noteNC{\alpha_n}$$
The case $k\ge1$, with $n-k=(m+1)d$ and projected system  $BC_d$
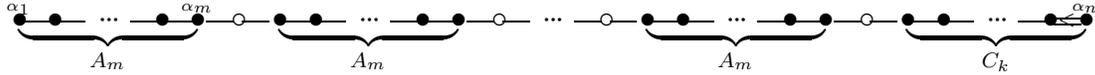
$$\underbrace{\disque\noteNC{\alpha_1}\th\disque\th\points\th\disque  \th\disque\noteNC{\alpha_m}}_{A_m} 
\th\cercle\th
\underbrace{\disque\th\disque\th\points\th\disque  \th\disque}_{A_m} \th\cercle\th
\points\th\cercle\th\underbrace{\disque\th\disque\th\points\th\disque  \th\disque}_{A_m} \th\cercle\th
 \underbrace{\disque\th\disque\th\points\th
\disque \thh\fgauche\disque\noteNC{\alpha_n}}_{C_k} $$
The case $k\ge1$, with $p=3m+2$ and $n-k=4(m+1)$, the projected system contains $A_2$
$$\underbrace{\disque\noteNC{\alpha_1}\th\disque\th\points\th\disque  \th\disque\noteNC{\alpha_m}}_{A_m} 
\th\cercle\th
\underbrace{\disque\th\disque\th\points\th\disque  \th\disque}_{A_p} \th\cercle\th
 \underbrace{\disque\th\disque\th\points\th
\disque \thh\fgauche\disque\noteNC{\alpha_n}}_{C_k} $$
\end{lemma}

\subsubsection{The case where $\Sigma$ is of type $D_n$}

With the notations analogous to the previous cases the roots are the $\pm e_i\pm e_j$ and
$\Delta$ is constituted of
$\alpha_i=e_i-e_{i+1}$ for $i\le n-1$
and of $\alpha_n=e_{n-1}+e_n$.

Case 1:  $\alpha_{n-1}=e_{n-1}-e_n$ and $\alpha_{n}=e_{n-1}+e_n$ are in $\Theta$
and the orthogonal complement of $\Theta$ admits the $e_i$ for $1\le i\le n-2$ as a basis.
The projection on the orthogonal of $e_n$ and $e_{n-1}$ contain the $\pm e_i\pm e_j$ along with the roots $\pm e_i$
for $i$ and $j$ between $1$ and $n-2$ obtained projecting the $\pm(e_i-e_n)$.
We, therefore, obtain the system $B_{n-2}$ already considered above.

Case 2: $\alpha_{n-1}=e_{n-1}-e_n$ is in $\Theta$ but $e_{n-1}+e_n$ is not.
As in the case of root system of type $B_n$ let us consider the three vectors $e_r$, $e_s$ and $e_t$ whose non-zero projections $\overline{e_r}$, $\overline{e_s}$ et $\overline{e_t}$
are distinct and associated to three components of $\Theta$ of type $A_m$, $A_p$ and $A_q$.
Once projected we find the $\pm\overline{e_{r}}\pm \overline{e_{s}}$ and $\pm\overline{e_{s}}\pm \overline{e_{t}}$.
We also have  $$2 \overline{e_{r}}=\overline{e_{r}}+\overline{e_{r+1}}=2\overline{e_{r+1}}$$ if $\alpha_{r}=e_r-e_{r+1}$
belongs to a connected component of $\Theta$.
Therefore $\Sigma_\Theta$ contains a root system of type $C_d$  if all the connected components of $\Theta$ are of the same cardinal $m$ with $n=d(m+1)$.

Case 2': analogous to the case 2 when exchanging $e_n$ with $-e_n$.

Case 3: Neither $\alpha_{n-1}=e_{n-1}-e_n$ nor $\alpha_{n}=e_{n-1}+e_n$ are in $\Theta$.
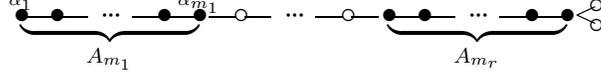
$$\underbrace{\disque\noteNC{\alpha_1}\th\disque\th\points\th\disque  \th\disque\noteNC{\alpha_{m_1}}}_{A_{m_1}} 
 \th\cercle\th
\points\th\cercle\th \underbrace{\disque\th\disque\th\points\th\disque 
\th\disque}_{A_{m_r}} \fourche\montepeu{}\descendpeu{}$$
We, therefore, have either an analogous situation to the one treated for $A_n$, or we consider
$\overline{\alpha}=\pm\overline{e_{n-1}}\pm \overline{e_{n}}$ and
$\overline{\beta}=\overline{e_s}\pm \overline{e_{n-1}}$.

In this case we have $e_n=\overline{e_n}$ and therefore with the now familiar notations
$$R=\frac{(1+(p+1))}{(1+\frac{p+1}{m+1})}
\qquad\hbox{and}\qquad  
C=(1+(p+1))\left(1+\frac{p+1}{m+1}\right)$$
Therefore

$$\frac{C}{R}=(1+\frac{p+1}{m+1})^2$$
which forces $R=1$ and $C=4$; thus $p=m=0$.
The existence of a system of maximal rank in the projection for a configuration of this sort forces   
$m_i=0$ for any $i$, that is $\Theta$ is empty, a case which is possible but trivial hence excluded
\textit{a priori}.

To sum up, we have proven the:

\begin{lemma}  
For a system of type $D$ the systems in the projection are of type $A$, $B$, $C$ or $D$.
If $\alpha_{n-1}=e_{n-1}-e_n$ and $\alpha_n=e_{n-1}+e_n$ are in $\Theta$ and if the others components of $\Theta$ 
are all of type $A_m$, the interval between two such components are of length one, with $n-k=(m+1)d$, then there exists a system of type $BC_d$ in $\Sigma_\Theta$.  
In the case 2 or 2', the projection contains a system of maximal rank of type $C_d$ if all the components are of  type $A_m$ and if $n=(m+1)d$.

The case 1: $D_k\subset\Theta$ with $k\ge 2$; we recover the case of $B_{n-k}$.
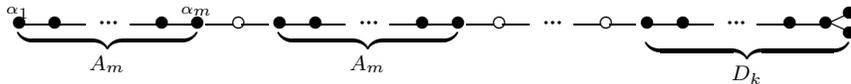
$$\underbrace{\disque\noteNC{\alpha_1}\th\disque\th\points\th\disque  \th\disque\noteNC{\alpha_m}}_{A_m} 
\th\cercle\th
\underbrace{\disque\th\disque\th\points\th\disque  \th\disque}_{A_m} \th\cercle\th
\points\th\cercle\th \underbrace{\disque\th\disque\th\points\th\disque 
\th\disque \fourcheD\montepeu{}\descendpeu{}}_{D_k}$$

The case 2 (or 2'):  
The projection contains a rank maximal system of type $C_d$ if all the components are of type $A_m$ and if $n=(m+1)d$.

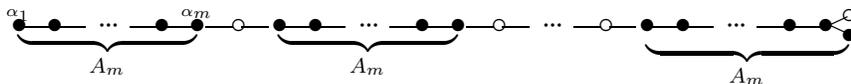
$$\underbrace{\disque\noteNC{\alpha_1}\th\disque\th\points\th\disque  \th\disque\noteNC{\alpha_m}}_{A_m} 
\th\cercle\th
\underbrace{\disque\th\disque\th\points\th\disque  \th\disque}_{A_m} \th\cercle\th
\points\th\cercle\th \underbrace{\disque\th\disque\th\points\th\disque 
\th\disque \fourched\montepeu{}\descendpeu{}}_{A_m}$$
\end{lemma}

\section{Exceptional root systems}

As opposed to the previous treatment in the context of classical root systems, the case of exceptional groups requires a case-by-case analysis which could not be completed by hand and motivated the use of SageMath. 

We run two codes, one answering Question 1 and the second Question 2. 
\newline
In both cases, we first compute the projections of roots in $\Sigma$ orthogonally to a subset of the basis $\Delta$. We use both the roots and Cartan matrix as defined in SageMath (which follows the numbering of Bourbaki).  
For $\Theta \subset \Delta$, denote $\Theta^\vee$ the set of its dual roots, and $C_\Theta$ its Cartan matrix. We need to find a certain combination, $\sum$, of the roots in $\Theta$, such that $t - \sum$ is orthogonal to all the coroots in $\Theta^\vee$. We also denote $v$ the vector $\big((t, \check{\alpha_i})\big)$.
Then the list of coefficients in the combination $\sum$ equals the vector $v$ multiplied by the transpose inverse of the Cartan matrix $C_\Theta$.
We can now write the equation for the projection: 
$$\hbox{proj}(t)= t - \sum_{\alpha_i \in \Theta}\hbox{coeff}[i-1]\alpha_i$$
where we substract all the simple roots in $\Theta$ weighted by the appropriate entry in the vector of coefficients "coeff".

\subsection{The code for Question 1}

\begin{description}
\item Once we have obtained all the projections of roots, we compute their squared norms (dot product). A rapid visual check of the different values appearing allow us to further count how many roots in the projections appear with a given norm. 
\item This is enough to exclude the possibility for many exceptional root systems to occur in the projection. See the proof of Theorem \ref{exc}.
\item In a few cases, we obtain a set of vectors (of a given norm) whose cardinal is the one of a smaller exceptional group, or two sets vectors with correct cardinal and different norms whose ratio corresponds to the ratio between long and short roots (for $F_4$ and $G_2$). Then we will proceed with the further step below. 
\end{description}

\subsection{Systematizing the algorithm of Question 1 to all $\Theta$'s} \label{syst}

The Dynkin diagrams of type $E$ exhibit many symmetries that one can use to reduce the number of cases of $\Theta$'s to consider. However, if one prefers to go over \textit{all} cases of $\Theta$'s, one can run a code similar to the code given in Appendix \ref{systappendix}. 

\begin{description}
\item One calculates the projections of roots and simple roots.
\item One converts them to vectors to eliminate the zero vectors, and append them to a list before removing duplicates. 
\begin{python}
projv= vector(proj)
        if projv!= zero_vector(8):
            listvec.append(proj)
            listvec2= uniq(listvec)
\end{python}

\item The first criteria of elimination of a given $\Theta$ is the number of vectors in the projection, for instance whenever we have less than 48 vectors, then $F_4$ will not occur. 

\item Then, one produces the list of squared norms (removing duplicates) of the projections of roots, and the list of squared norms (with duplicates) of the projections of simple roots. Then one notices that only very few cases of $\Theta$'s give appropriate ratio of squared lengths (such as 3 if one is looking for a $G_2$ in the projection). 
Further, one counts the number of projections of roots with a given length, removing further cases (if there are not enough vectors of a given norm to allow for a $G_2$ or a $F_4$ to occur). 

\begin{python}
if len(listvec2)>47:
        for v in listvec2:
            n= len(listnorm)
            if v.dot_product(v)==listnorm[0]:
                listv1.append(v)
                listv1 = uniq(listv1)    
            if v.dot_product(v)==listnorm[1]:
                listv2.append(v)
                listv2 = uniq(listv2)
            if v.dot_product(v)==listnorm[n-2]:
                listv3.append(v)
                listv3 = uniq(listv3)
            if v.dot_product(v)==listnorm[n-1]:
                listv5.append(v)
                listv5 = uniq(listv5)
    lst.append(len(listv1))
    lst.append(len(listv2))
    lst.append(len(listv3))
    lst.append(len(listv5))
\end{python}

\item Finally, provided we have a correct ratio of squared norms, and each norm appears as the norm of sufficiently enough vectors, we can now look at the projections of the simple roots. In many cases, the projections of the simple roots do not have the norms with appropriate ratio. 
\end{description}

\subsection{The code for Question 2}

After implementing the code for Question 1, we are left with only a few cases where we need to verify if the projection of the simple roots is a basis of the system $X$ appearing in the projection. If it is not, we can still identify its basis by running a loop over all combinations of $|\Delta_X|$ projected vectors. 

\begin{description}
\item Assume there exists an exceptional root system of rank $d$, $X$, in $a_\Theta$ (e.g $E_6$), and let $W_X$ be the Weyl group of this exceptional root system. If there exists a set of vectors in the projection which constitutes a basis of $X$, then the Weyl group action on this set should generate all its roots, i.e should generate a set of cardinal $C_{X}$, the cardinal of $X$.
\item We compute the projections of simple roots separately, and check, using the "letsapply" function (see below), if the Weyl group constructed using these roots generate the expected number of roots. 
\item Verify if the basis' vectors have correct norms and products.
\item One can also compare the set of vectors generated by the action of the Weyl group $W_X$ on $X$ to the initial set of projections (we did in the cases $(E_8,G_2)$, $(E_8,F_4)$, $(F_4,G_2)$).
\end{description}

We now explain the different functions in this code:
\begin{itemize}
\item The "gensi" function defines elementary symmetries $(s_{\alpha_i})$.
\begin{python}
def gensi(t, projai):
    myLone = projai
    result = t- 2*myLone.dot_product(t)/(projai.dot_product(projai))*myLone
    return result
\end{python}
\item We apply the "decompose" function to the Weyl group (already defined in SageMath), so that we can apply successively each elementary symmetry to a potential basis of the expected root system. 

\begin{python}
def decompose(g, mylength):
    # g is a str(aa) where aa belongs to list(WeylGroup ...)
    listg = [g[3*i]+ g[3*i+1] for i in range(0, mylength)]
    return listg
\end{python}

\item The "letsapply" is just doing that: letting each element of the symbolic Weyl group (decomposed into elementary symmetries) acts on the potential basis of the root system.

\begin{python}
def letsapply(listg, alist, p1, p2):
    mylistg = listg
    mylistg.reverse()
    for si in mylistg:
        if si == 's1':
            alist = map(lambda w: gensi(w, p1), alist)
        elif si == 's2':
            alist = map(lambda w: gensi(w, p2), alist)
        else:
            print("Wrong input in letsapply!!")
    alistt = [tuple(v) for v in alist]
    alistts = set(alistt)

    return alistts
\end{python}

\item Note that the Weyl group is also constructed using the vectors in this "potential basis". 
\item We take randomly $n$ vectors in the projection for the "potential basis", construct the Weyl group from them, and let the elements of the Weyl group acts on them to generate, ideally, a number of vectors which is exactly the cardinal of the "expected root system". 

For instance, if "myP" denotes the set of 
\begin{python}
 proja1= myP[0]
 proja2= myP[1]
 myLel= [proja1, proja2]
 res = set(tuple(v) for v in myLel)
 for listg in listWdecomposed:
    res = res.union(letsapply(listg, myLel, proja1, proja2))
\end{python}

\end{itemize}

\begin{rmk}
The current code needs to be adapted (with different $\Theta$'s, expected root system in the projection, and cardinal of this expected root system) to each case.
\end{rmk}

\subsection{Proof of the main theorems for exceptional groups}

The results of Theorem \ref{exc} rely on the following Lemma:
\begin{lemma}\label{thetaone}
Let $\Sigma$ be an irreducible root system.
Any two subsets $\Theta$ and $\Theta'$, each constituted of only one root of $\Sigma$, of the same length, are conjugated under the
Weyl group $W$. If w is a Weyl group element sending $\Theta$ on $\Theta'$, we have $\Sigma_{\Theta'} = w(\Sigma_{\Theta})$.
\end{lemma}

\begin{proof}
First by an equivalent of the incomplete basis theorem for root systems (see \cite{bourbaki}, Chap VI, 1, Proposition 24), it is always possible to complete any root to a basis of the root system. Hence it is enough to consider the case of simple roots. \\
By a classical lemma (see for instance Lemma C in \cite{humphreys}, III, 10), all roots of the same length are conjugate under $W$.
Let us consider any two basis' roots of the same length, $\alpha$, and $\beta = w(\alpha)$.
Let us assume we project all vectors in $\Sigma$ orthogonally to $\Theta= \{\alpha\}$, and $\Theta'= \{\beta\}$. 
Since the Weyl group is a subgroup of the isometry group of the root system, it preserves lengths and angles.
Therefore applying first $w\in W$ to all roots and projecting with respect to $\beta$ yields the same result as applying $w\in W$ to  $\Sigma_{\Theta}$.
\end{proof}

\begin{rmk}\label{oneremark}
Under the conditions of Lemma \ref{thetaone}, if $\Sigma_{\Theta}$ contains a maximal root system of rank $d$, then it doesn't depend on the choice of the root generating $\Theta$; It is enough to determine $\Sigma_{\Theta}$ for any root in $\Sigma$. Indeed the ratios and angles between roots in the maximal root system occurring in $\Sigma_{\Theta}$ and $w(\Sigma_{\Theta})$ are the same.
\end{rmk}

\begin{proof}[Proof of Theorem \ref{exc}]
Let us start with $\Sigma = E_8$. We consider first the projection orthogonal to one simple root. By Lemma \ref{thetaone}, we can choose any of them, and the result will be the same. Then, using SageMath we observe that two different norms appear in the projection orthogonal to any simple root. In particular, there are exactly 126 vectors of the same norm, which is the number of roots of $E_7$. The existence of $E_7$ can also be seen using the basis given in the Appendix \ref{EE}, when taking the projection orthogonal to $\alpha_8$
$${\cercle\noteSC{\alpha_2}\th\cercle\noteSC{\alpha_3}}
\th
\ctvb\noteSC{\alpha_4}{\gnoteN{\alpha_1}{20.5}}
\th
{\cercle\noteSC{\alpha_5}\th\cercle\noteSC{\alpha_6}\th\cercle\noteSC{\alpha_7}\th\disque\noteSC{\alpha_8}}$$

We obtain $\alpha_1 = \frac{1}{2}[e_0+ e_1+ e_2+e_3-e_4-e_5- e_6-e_7]$ and the $\alpha_i$ with $2 \leq i \leq 6$ which generates the $E_6$, but also $\beta  = -\frac{1}{2}[e_0- e_1- e_2-e_3-e_4-e_5+ e_6-e_7]$. The Dynkin diagram associated to $(\alpha_1, \ldots, \alpha_6, \beta)$ is the one of $E_7$. It also possible to check that the root system $E_7$, a subsystem in $E_8$, happens to be orthogonal to some root in $E_8$, and again conclude by Lemma \ref{thetaone}. 
It could be tempting to believe an $E_6$ will appear in the projection orthogonal to some root in $E_7$, but this time, there are not enough vectors of a given norm in the projection, respectively 60 and 62, whereas $E_6$ has 72 roots. 
It is however possible when starting from $E_8$ to find 72 roots of the same norm, prefiguring $E_6$, in the projection orthogonal to two simple roots in three cases (three different $\Theta$). Since $E_6$ is also a subsystem in $E_8$, it is possible to check that it is orthogonal to two specific roots of $E_8$. With the same method, we can verify that no root system of type $F_4$ appear in $E_7$ and $E_6$ (for all possible combination of simple roots in $\Theta$), but one $F_4$ appears in $E_8$. 
The root system of $G_2$, of smaller size is prefigured in the projection in all exceptional root systems, albeit only for a unique $\Theta$ for type $E$ root systems, and for two in $F_4$. In the proof of the next theorem, we check that not only those sets have the correct cardinality, but that they are indeed root systems. 
\end{proof}

\begin{proof}[Proof of Theorem \ref{exc2}]
Using the results of Theorem \ref{exc}, we have only a few cases to verify here. Since we have taken the projection orthogonal to some simple root of $E_8$ to get $E_7$, it is clear that the projection of the adjacent simple root(s) in the Dynkin diagram will not have the same norm as the projection of other simple roots rendering impossible for this $E_7$ to have a basis constituted of projections of simple roots. Using the Question 2's SageMath code, we can also show that $E_6$ won't either have a basis constituted of projections of simple roots. However, we can check with this code that the $G_2$ and $F_4$ (for the relevant $\Theta$) are generated by projections of simple roots. Similarly, the $G_2$ appearing in the projections of $F_4$, $E_6$ and $E_7$, for the specific choices of $\Theta$ obtained in Theorem \ref{exc} have a basis constituted of projections of simple roots.  
\end{proof}

\subsection{The case of exceptional root system appearing in a reducible subsystem of $\Sigma_\Theta$}

\begin{proposition} \label{exc3}
Let $\Sigma$ be an irreducible root system of exceptional type with basis $\Delta$. Let $\Theta$ be a subset of $\Delta$ and $d=|\Delta -\Theta|$. In the following cases, and only for those cases, $\Sigma_{\Theta}$ contains a reducible root system of rank $d$ with at least one irreducible component of exceptional type, and whose basis is constituted of projections of simple roots in $\Delta$.
\begin{itemize}
\item Let $\Sigma$ be of type $E_8$. If $\Theta$ is either of $\{\alpha_1,\alpha_3,\alpha_5, \alpha_6, \alpha_8\}$ or $\{\alpha_2,\alpha_4,\alpha_5, \alpha_6, \alpha_7\}$, then $\Sigma_{\Theta}$ contains a system of type $G_2\times A_1$. If $\Theta=\{\alpha_2, \alpha_5,\alpha_7\}$, $\Sigma_{\Theta}$ contains a system of type $F_4\times A_1$. We can also find $G_2\times A_1\times A_1$ if $\Theta= \{\alpha_1,\alpha_3,\alpha_5, \alpha_6\}$. 
\item Let $\Sigma$ be of type $E_7$. If $\Theta$ is either of $\{\alpha_1,\alpha_3,\alpha_5, \alpha_6\}$, $\Sigma_{\Theta}$ contains a system of type $G_2\times A_1$. 
 \end{itemize}
\end{proposition}

\begin{proof}
In this proposition, we only focus on the cases where the basis of our potential root system in the projection is constituted of projections of simple roots. We consider all the different possibilities of having a product of exceptional root systems of rank $d$ (such as $F_4\times G_2$, or $G_2\times G_2$) in $\Sigma_\Theta$ whenever $\Sigma$ is exceptional, or a product of an exceptional root system with a (or many) classical one. In particular, we consider the possibilities of $G_2 \times A_n$ (resp. $F_4 \times A_n$) for various $n$ to occur. We first use the algorithm given in Subsection \ref{syst} to eliminate most of the $\Theta$'s (satisfying the condition $\rk(\Sigma)-\rk(\Theta)=d$). If $\Sigma$ is of type $E_7$ or $E_6$, letting $d$ varies, we did not find any $\Theta$ letting a $G_2$ or $F_4$ appear except one. Most of the time the reason being that the ratio of squared norms were neither 3 nor 2, or because such squared norms appeared for too few vectors. In the case of $E_7$, we find two set of roots potentially leading to a $G_2\times A_1$, $\{\alpha_1,\alpha_3,\alpha_5, \alpha_6\}$ and $\{\alpha_2,\alpha_4,\alpha_6, \alpha_7\}$. Further examination using the code of Question 2 shows that only the first set leads to the appearance of $G_2\times A_1$ in the projection. 

Now considering $\Sigma=E_8$, we considered the case of $G_2\times G_2$. Only a few $\Theta$'s gave squared lengths whose ratio was 3. In those cases, the squared lengths were [2, 2/3, 4/3] with enough vectors (12 at least) for the squared norms 2/3 and 2. However, if two $G_2$ appear, we need two projections of simple roots of each norm, but in all cases one projection (among the four projections of simple roots) was of squared norm 4/3. For the same reason, $G_2\times A_2$ doesn't appear. This could also be verified using the code for Question 2 in the doubtful cases. Finally, in the few cases listed in the statement of the theorem, we verified their occurrence using the code for Question 2. 
\end{proof}

\appendix
\section{A basis different from Bourbaki's for type $E$ root systems} \label{EE}

We recall the $E$ basis as proposed by Jean-Pierre Labesse (see \cite{JIMJ} and an unpublished note): the details are given below (see also the remark \ref{Labesseconvention}). We use it to understand how $E_7$ is obtained in the projection of $E_8$. 

\subsubsection{The case $E_6$}

We consider the euclidean space $\widetilde{V}$ of dimension 8, equipped with a orthonormal basis indexed by the elements 
of $\mathbb Z/8\mathbb Z$
$$\{e_0,e_1,\cdots,e_6,e_7\}$$ such that $e_0$ will sometimes be denoted $e_8$.
The roots of $E_6$ are the roots in $E_7$ orthogonal to $e_7 - e_8 = -(e_7 + e_0)$ (see the definitions of $E_8$ and $E_7$ in the next subsections).
They are of the following form: 
\begin{itemize}
\item $\pm(e_i-e_j)$ for $1\le i<j\le6$ or $i=0$ and $j=7$.
\item $\pm\frac{1}{ 2}[(e_0-e_7)\pm e_1\pm e_2\cdots\pm e_6]$ 
\end{itemize} with the same number of 
 $+$ and $-$ sign in the bracket.
A system of simple roots is given by 
$$\alpha_{1} =\frac{1}{2}[e_0+ e_1+ e_2+e_3-e_4-e_5- e_6-e_7]\qquad\hbox{and}\qquad
\alpha_{i+1}=[e_{i+1}-e_i]\qquad\hbox{for $1\le i\le5$}\,\,.
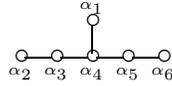
$$
$${\cercle\noteSC{\alpha_2}\th\cercle\noteSC{\alpha_3}}
\th
\ctvb\noteSC{\alpha_4}{\gnoteN{\alpha_1}{20.5}}
\th
{\cercle\noteSC{\alpha_5}\th\cercle\noteSC{\alpha_6}}$$

\begin{rmk} \label{Labesseconvention}
This depiction is different from the one given in Bourbaki: we have used a subsystem of the system $E_8$ as defined by Bourbaki, except that $\epsilon_8$ is here $e_0$ and that we have an order -and therefore simple roots- which is (are) different(s).
In particular, in our convention the roles of  $\alpha_1$ and $\alpha_2$ in the Dynkin diagram are inverted.
The correspondence is the following:\par
\smallskip
\hfil  \kern -10pt Our notation $\qquad\longleftrightarrow$\quad
 Bourbaki's notation \hfill\par\noindent
$\alpha_1 =\frac{1}{ 2}[e_0+ e_1+ e_2+e_3-e_4-e_5- e_6-e_7]
 \longleftrightarrow\quad\alpha_2=\epsilon_1+\epsilon_2$\par\medskip
 \hskip 125pt
 $\alpha_2=e_{2}-e_1\longleftrightarrow\quad \alpha_1=
\frac{1}{ 2}[\epsilon_1
 -\epsilon_2-\epsilon_3-\epsilon_4-\epsilon_5-\epsilon_6-\epsilon_7+\epsilon_8]$\par
 \hskip 110 pt  $\alpha_{i+1}=e_{i+1}-e_i\longleftrightarrow\quad
\alpha_{i+1}=\epsilon_{i}-\epsilon_{i-1}\qquad\hbox{for}\qquad 2\le i\le 5$
\par\noindent
With our writing, it is easily seen that there exists an automorphism
$\theta(e_i)=-e_{(7-i)}$; sending $\alpha_{i+1}$ on $\alpha_{7-i}$ for $1\le i\le5$ and it fixes $\alpha_{1} $ and $\alpha_4$.\\
\end{rmk}

\subsubsection{The case $E_8$} \label{e8}
The positive roots are of the following form:
\begin{itemize}
\item { } $\pm e_i\pm e_j$ for $0\leq i<j \leq 7$.
\item { } $\frac{1}{2}[\pm e_0\pm e_1\pm e_2\cdots\pm e_6\pm e_7]$ with an even number of negative signs.
\end{itemize}
A system of simple roots is given by:

$$\alpha_{1} =\frac{1}{ 2}[e_0+e_1+e_2+e_3-e_4-e_5-e_6-e_7]\qquad\hbox{and}\qquad
\alpha_{i}= e_{i}-e_{i-1}\qquad\hbox{for $2\le i\le 8$}$$

$${\cercle\noteSC{\alpha_2}\th\cercle\noteSC{\alpha_3}}
\th
\ctvb\noteSC{\alpha_4}{\gnoteN{\alpha_1}{20.5}}
\th
{\cercle\noteSC{\alpha_5}\th\cercle\noteSC{\alpha_6}\th\cercle\noteSC{\alpha_7}\th\cercle\noteSC{\alpha_8}}$$ 


\section{Example of the code of Question 1 for $(E_6, F_4)$}
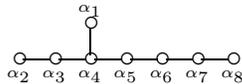
\begin{python}
e = RootSystem(['E',6]).ambient_space()
Roots = e.roots()
C = CartanMatrix(['E', 6])
print(C)

S = C.subtype([1,2])
print(S)
C = S.transpose()
print(C.inverse())

def prod(x, ai):
    result = 2*x.dot_product(ai)/ai.dot_product(ai)
    return result

listvec=[]
listvec2=[]
listnorm=[]
listv1=[]
listv2=[]
listv3=[]

for x in Roots:
    v1 = prod(x, e.simple_root(1))
    v2 = prod(x, e.simple_root(2))
    v= vector((v1, v2))
    coeff= v*C.inverse()
    proj= x-coeff[0]*e.simple_root(1)-coeff[1]*e.simple_root(2)
    projv= vector(proj)
    if projv != zero_vector(6):
        listvec.append(proj)
        listvec2= uniq(listvec)
        for v in listvec2:
            print(v.dot_product(v))
            if v.dot_product(v)==2:
                listv1.append(v)
                listv1 = uniq(listv1)
            if v.dot_product(v)==3/2:
                listv2.append(v)
                listv2 = uniq(listv2)                   
            if v.dot_product(v)==1:
                listv3.append(v)
                listv3 = uniq(listv3)
                if (len(listv1)+len(listv2)+len(listv3))>20:
                    print(len(listv1))
                    print(len(listv2))
                    print(len(listv3))
                    print(len(listvec2))
                    
\end{python}

\newpage
\section{Systematizing the algorithm of Question 1}
\label{systappendix}
\begin{python}
e = RootSystem(['E',7]).ambient_space()
Roots = e.roots()
Simples = e.simple_roots()
C = CartanMatrix(['E', 7])

def prod(x, ai):
    result = 2*x.dot_product(ai)/ai.dot_product(ai)
    return result

RootsN = [1,2,3,4,5,6,7]
for index in Arrangements(RootsN, 4):
    listnorm=[]
    listvec=[]
    listvec2=[]
    listv1=[]
    listv2=[]
    listv3=[]
    listv4=[]
    listv5=[]
    listsimples=[]
    listnorm2=[]
    lst=[]
    i= index[0]
    j= index[1]
    k= index[2]
    l= index[3]
    J= [i,j,k,l]
    with open("C:/Users/.../E7G2A1.txt", "a") as f:
        f.write("G2A1")
        f.write("\n")
        f.write(str(J))
        f.write("\n")
    S = C.subtype([i,j,k,l])
    Ct = S.transpose()
    
    for x in Simples:
        v1 = prod(x, e.simple_root(i))
        v2 = prod(x, e.simple_root(j))
        v3 = prod(x, e.simple_root(k))
        v4 = prod(x, e.simple_root(l))
        v= vector((v1, v2, v3, v4))
        coeff= v*Ct.inverse()
        proj= x-coeff[0]*e.simple_root(i)-coeff[1]*e.simple_root(j)-coeff[2]*e.simple_root(k)-coeff[3]*e.simple_root(l)
        projv= vector(proj)
        if projv!= zero_vector(8):
            listsimples.append(proj)
            listsimples= uniq(listsimples)
    with open("C:/Users/Sarah/Documents/E7G2A1.txt", "a") as f:
        f.write("list simples")
        f.write(str((listsimples)))
    for v in listsimples:
        R= v.dot_product(v)
        listnorm2.append(R)
    with open("C:/Users/.../E7G2A1.txt", "a") as f:
        f.write("\n")
        f.write(str((listnorm2)))     
        f.write("\n")
    
    for x in Roots:
        v1 = prod(x, e.simple_root(i))
        v2 = prod(x, e.simple_root(j))
        v3 = prod(x, e.simple_root(k))
        v4 = prod(x, e.simple_root(l))
        v= vector((v1, v2, v3, v4))
        coeff= v*Ct.inverse()
        proj= x-coeff[0]*e.simple_root(i)-coeff[1]*e.simple_root(j)-coeff[2]*e.simple_root(k)-coeff[3]*e.simple_root(l)
        projv= vector(proj)
        if projv!= zero_vector(8):
            listvec.append(proj)
            listvec2= uniq(listvec)
    with open("C:/Users/.../E7G2A1.txt", "a") as f:
        f.write(str(len(listvec2)))
        f.write("\n")
    for v in listvec2:
        R= v.dot_product(v)
        listnorm.append(R)
        listnorm= uniq(listnorm)
    with open("C:/Users/Sarah/Documents/E7G2A1.txt", "a") as f:
        f.write("\n")
        f.write("listnorm")
        f.write((str((listnorm))))
        f.write("\n")
    if len(listvec2)>12:
        for v in listvec2:
            n= len(listnorm)
            if v.dot_product(v)==listnorm[0]:
                listv1.append(v)
                listv1 = uniq(listv1)    
            if v.dot_product(v)==listnorm[1]:
                listv2.append(v)
                listv2 = uniq(listv2)
            if v.dot_product(v)==listnorm[n-3]:
                listv3.append(v)
                listv3 = uniq(listv3)
            if v.dot_product(v)==listnorm[n-2]:
                listv4.append(v)
                listv4 = uniq(listv4)
            if v.dot_product(v)==listnorm[n-1]:
                listv5.append(v)
                listv5 = uniq(listv5)  
    lst.append(len(listv1))
    lst.append(len(listv2))
    lst.append(len(listv3))
    lst.append(len(listv4))
    lst.append(len(listv5))
    with open("C:/Users/.../E7G2A1.txt", "a") as f:
        f.write(str(lst))

\end{python}

\newpage
\section{Example of the code of Question 2 for $(E_7, G_2)$}
\begin{python}
C = CartanMatrix(['E', 7])
print(C)
S = C.subtype([2,4,5,6,7])
print(S)
C = S.transpose()
print(C.inverse())

e = RootSystem(['E',7]).ambient_space()

def prod(x, ai):
    result = 2*x.dot_product(ai)/ai.dot_product(ai)
    return result

def proj(t, C):
    v1 = prod(t, e.simple_root(2))
    v2 = prod(t, e.simple_root(4))
    v3 = prod(t, e.simple_root(5))
    v4 = prod(t, e.simple_root(6))
    v5 = prod(t, e.simple_root(7))
    v= vector((v1, v2, v3, v4, v5))
    coeff= v*C.inverse()
    proj= t-coeff[0]*e.simple_root(2)-coeff[1]*e.simple_root(4)-coeff[2]*e.simple_root(5)-coeff[3]*e.simple_root(6)-coeff[4]*e.simple_root(7)
    return proj

def gensi(t, projai):
    myLone = projai
    result = t- 2*myLone.dot_product(t)/(projai.dot_product(projai))*myLone
    return result

def decompose(g, mylength):
    # g is a str(aa) where aa belongs to list(WeylGroup ...)
    listg = [g[3*i]+ g[3*i+1] for i in range(0, mylength)]
    return listg

def letsapply(listg, alist, p1, p2):
    mylistg = listg
    mylistg.reverse()
    for si in mylistg:
        if si == 's1':
            alist = map(lambda w: gensi(w, p1), alist)
        elif si == 's2':
            alist = map(lambda w: gensi(w, p2), alist)
        else:
            print("Wrong input in letsapply!!")
    alistt = [tuple(v) for v in alist]
    alistts = set(alistt)

    return alistts

def main():
    W = WeylGroup(['G', 2], prefix="s")
    listW = list(W)
    # We remove the first element which is 1:
    listW.pop(0)
    listWdecomposed = [decompose(str(g), g.length()) for g in listW]
    myP=[]
    e = RootSystem(['E',7]).ambient_space()
    Roots = e.roots()
    Simples = e.simple_roots()
    print(len(Simples))
    ## get all projections of roots in E7
    myL = [proj(x, C) for x in Roots]
    # remove duplicates
    myL= uniq(myL)
    myL = [vector(v) for v in myL]
    myL = [pp for pp in myL if pp != zero_vector(8)]
    if len(myL)<12:
        return 0
        print("myL too small")
   # deal with the projections of simple roots 
    for x in Simples:
        v1 = prod(x, e.simple_root(2))
        v2 = prod(x, e.simple_root(4))
        v3 = prod(x, e.simple_root(5))
        v4 = prod(x, e.simple_root(6))
        v5 = prod(x, e.simple_root(7))
        v= vector((v1, v2, v3, v4, v5))
        coeff= v*C.inverse()
        pr= x-coeff[0]*e.simple_root(2)-coeff[1]*e.simple_root(4)-coeff[2]*e.simple_root(5)-coeff[3]*e.simple_root(6)-coeff[4]*e.simple_root(7)
        prv= vector(pr)
        if prv != zero_vector(8):  
            myP.append(pr)
            print(myP)
            print(len(myL))
            for a in myP:
                print(a.dot_product(a))
            myP = [vector(v) for v in myP]
            myP = [pp for pp in myP if pp != zero_vector(8)]
            # can write some info in some external file
        with open("E7G2.txt", "a") as f:
            f.write("myP and len(myL)")
            f.write(str(myP))
            f.write(str(len(myL)))
    
    # Define the two projections of simple roots that will be used to construct the Weyl group
    proja1= myP[0]
    proja2= myP[1]
    myLel= [proja1, proja2]
    print(myLel)
    print("End of precomputations, beginning of the loop.")
    res = set(tuple(v) for v in myLel)
    ShouldTakeNextmyLel = False
    for listg in listWdecomposed:
            res = res.union(letsapply(listg, myLel, proja1, proja2))
            print(res)
            print(len(res))
            # discriminate depending on the size of the set generated by the Weyl group acting on the two vectors
            if len(res) > 12:
                ShouldTakeNextmyLel = True
                break # No use continuing with this myLel
            if len(res) == 12: 
                res = [vector(v) for v in res]
                res = list(res)
                for x in res:
                    print(x.dot_product(x))
                LLL =[]
                lst=[]
                for a in myLel:
                    for b in myLel:
                        R = prod(a,b)
                        lst.append(R)
                        print(lst)

    if ShouldTakeNextmyLel:
        next

    return 0

main()
\end{python}

\addresseshere

\bibliographystyle{plain}
\bibliography{PRSR2}
\bigskip

\end{document}